\documentclass[a4paper,12pt,wlayout]{article}

\usepackage{array} 
\usepackage{amssymb}
\usepackage{amsmath} 
\usepackage{amsbsy} 
\usepackage{epsfig}
\usepackage{mathrsfs} 
\usepackage{graphics}
\usepackage{graphicx}
\usepackage[sort&compress,numbers]{natbib}
\usepackage{pstricks}
\usepackage{float}
\usepackage{amsthm}
\usepackage{slashed}
\usepackage{a4wide}

\newcommand{\lb}{\label}
\newcommand{\go}{\rightarrow}

\newcommand{\ee}{\end{equation}}
\newcommand{\be}{\begin{equation}}
\newcommand{\bea}{\begin{eqnarray}}
\newcommand{\eea}{\end{eqnarray}}
\newcommand{\sbea}{\begin{subequations}\begin{eqnarray}}
\newcommand{\seea}{\end{eqnarray}\end{subequations}} 
\newcommand{\ees}{\end{equation*}}
\newcommand{\bes}{\begin{equation*}}
\newcommand{\beas}{\begin{eqnarray*}}
\newcommand{\eeas}{\end{eqnarray*}}

\newcommand{\rf}[1]{(\ref{#1})}

\theoremstyle{plain}
\newtheorem{theorem}{\bf Theorem}[section]

\newtheorem{corollary}[theorem]{\bf Corollary}

\theoremstyle{remark}
\newtheorem{remark}[theorem]{\bf Remark}
\newtheorem{example}[theorem]{\bf Example}

\begin{document}
\title{Long-time behaviour of solutions to a singular heat equation with an application to hydrodynamics}

\author{Georgy Kitavtsev\thanks{Mathematical Institute, University of Oxford,
Woodstock Road, Oxford OX2 6GG, UK}$\ $ and Roman M. Taranets\thanks{Institute of Applied Mathematics and Mechanics of the NASU, Dobrovolskogo Str. 1, 84100, Sloviansk, Ukraine}}

\maketitle

\begin{abstract}
In this paper, we extend the results of~\cite{FKT18} by proving exponential asymptotic $H^1$-convergence
of solutions to a one-dimensional singular heat equation with $L^2$-source
term that describe evolution of viscous thin liquid sheets while considered in the Lagrange
coordinates. Furthermore, we extend this asymptotic convergence result to the
case of a time inhomogeneous source. This study has also independent interest for the porous medium equation theory.

\end{abstract}

\section{Introduction}\label{A}

In this paper, we study the long-time behaviour of solutions
to the initial boundary value problem for a singular heat equation in the presence of an external source term:
\begin{equation}\label{eq-1}
u_t = \nu (u^{-2}u_x)_x  + f(x,t) \text{ in } Q_T:=(0,1)\times(0,T),
\end{equation}
\begin{equation}\label{eq-2}
u_x(0,t)=u_x(1,t)=0\quad\text{for }t\in (0,\,T),
\end{equation}
\begin{equation}\label{eq-3}
u(x,0)=u_0(x)\quad\text{for }x\in(0,\,1),
\end{equation}
with $T > 0$ and $\nu > 0$.
Additionally, we assume zero mean of the external force:
\be
\int_0^1 f(x,t)\,dx=0\quad\text{for all } t\in(0,\,T).
\lb{f_mean}
\ee
Let us point out that equation \rf{eq-1} without source term $f(x,t)$ was considered first in~\cite{St51} and is related to the porous-media equations~\cite{Va07}:
\be
u_t=\Delta u^m,\quad m\geqslant-1.
\lb{PME}
\ee
Indeed, for the critical exponent $m=-1$ the latter equation coincides with \rf{eq-1} after the reversion of the time variable.
There exists also a relation between \rf{eq-1} and the classical heat equation through the Lie-B\"acklund transform
investigated in~\cite{Ro79,BK80,Ki90} that allows to find special exact solutions to the former. We do not pursue this
approach in this study. Let us also mention recent study~\cite{SSZ16} of solutions to \rf{eq-1} with a time dependent source term
$f(x,t)$ and an application to the Penrose-Fife phase field system~\cite{PF90,Zh92}, in which \rf{eq-1} similar to the original study~\cite{St51} 
appears as an equation for temperature distribution.

By contrast, the motivation of this study for considering problem \rf{eq-1}-\rf{f_mean} comes from the recent connection found in~\cite{FKT18}
of it to the degenerate parabolic system 
\begin{subequations}
\label{SSM}
\begin{align}
v_t+vv_y&=\frac{\nu}{h}[hv_y]_y,
\label{SSM1}\\
h_t&= -\left(h v\right)_y,
\label{SSM2}
\end{align}
\end{subequations}
describing evolution of the free surface $h(y,t)$ and averaged lateral
velocity $v(y,t)$ of a 2D viscous thin sheet in the absence of surface tension with $(y,t)\in Q_T$ and
considered with the following boundary conditions:
\be
h_y(0,t)=h_y(1,t)=v(0,t)=v(1,t)=0\quad\text{for}\quad t\in(0,\,T),
\lb{BC}
\ee
and initial data having positive height:
\bes
h(y,0)=h_0(y)>0,\  v(y,0)=v_0(y)\quad\text{for}\quad y\in(0,\,1).
\ees
We note that system \rf{SSM} is also a special example of the viscous
Saint-Venant or shallow-water equations that were derived and studied
extensively in the last three decades, see~\cite{ODB97,GP01,Br09,BN11,EF15} and references therein.

It was shown in~\cite[section 3]{FKT18} that  system \rf{SSM} and equation \rf{eq-1} considered with the special time-independent source term:
\be
f(x,t)=f_0(x):=\frac{M}{h_0(y(x,0))}\left[v_0(y(x,0))+\nu\frac{h_{0,y}(y(x,0))}{h_0(y(x,0))}\right]_y
\lb{f_def}
\ee
are related by the Eulerian-to-Lagrangian coordinate transformation: 
\be
y_x(x,t)=\frac{M}{h(y(x,t),t)},\quad y_t(s,t)=v(y(x,t),t),\quad y(x,0)=M\int_0^x\frac{ds}{h_0(s)}
\lb{x_s}
\ee
with constant $M>0$ denoting the total conserved mass of the viscous sheet:
\be
M(t):=\int_0^1 h(y,t)\,dy=M\quad\text{for all } t\in(0,\,T).
\lb{M_consv}
\ee
Additionally, one defines function
\be
u(x,t):=\frac{M}{h(y(x,t),t)}.
\lb{u_h_rel}
\ee
The conservation of mass \rf{M_consv} is ensured by boundary conditions \rf{BC} and pertains according to transformation \rf{x_s} also for the
corresponding solutions to problem \rf{eq-1}-\rf{f_mean} with \rf{f_def} in
the following form:
\be
\int_0^1u(x,t)\,dx=1\quad\text{for all } t\in(0,\,T).
\lb{cons_mass}
\ee
In~\cite[Theorem 4.1]{FKT18} the authors have shown that without the source term $f(x,t)\equiv 0$ solutions to \rf{eq-1}-\rf{f_mean}
asymptotically decay to the constant profile $u_\infty=1$ in $H^1$-norm. Besides, basing on numerical simulations of system \rf{SSM},
it was conjectured in~\cite{FKT18} (see Remark 4.3 and Fig. 2 there) that in the case of time-independent source $f(x,t)=f_0(x)$
solutions to \rf{eq-1}-\rf{f_mean} converge to a unique non-constant positive
steady state $u_\infty$ solving the stationary version of
\rf{eq-1}--\rf{eq-3}, namely,
\be
u_\infty(x) = \frac{\nu}{\int \limits_0^x { \int \limits_0^y { f(s)\, ds}dy  } + C_\nu}>0
\lb{LP}
\ee
with $C_\nu$ being a unique root, due to \rf{cons_mass}, of the following equation:
\be
\int_0^1\left[\int\limits_0^x {\int\limits_0^y { f(s)\, ds}dy  }+C_\nu\right]^{-1}\,dx=\nu^{-1}.
\lb{C}
\ee
Indeed, for $f\in L^2(0,\,1)$ the double integral in \rf{LP}-\rf{C} is
$C^1[0,\,1]$ function and, therefore, for any $\nu>0$ positive function $u_\infty\in
H^2(0,\,1)$ and constant $C_\nu$ are uniquely defined.

Moreover, numerical simulations of~\cite{FKT18} indicated that the corresponding solutions to the shallow-water system \rf{SSM}
defined through transformation \rf{x_s} converge as time goes to infinity to the corresponding limit profiles:
\be
y_\infty(x)=\int_0^x u_\infty(s)\,ds,\quad h_\infty(y_\infty(x))=\frac{M}{u_\infty(x)},\quad v_\infty(y_\infty(x))=0\quad\text{for }x\in(0,1),
\lb{final_non_flat} 
\ee
which are defined explicitly by the initial data $(h_0,v_0)$ as, due to \rf{f_def}, one has
\be
u_\infty(x) = \frac{\nu}{\int_0^x v_0((y(s,0))\,ds+\frac{\nu}{M}(h_0(y(x,0)-h_0(0))+C_\nu}.
\lb{u_infty}
\ee
This intriguing result implies selection of a single attracting steady-state \rf{final_non_flat} to system \rf{SSM} depending only
on the given initial data $(h_0,v_0)$ through $u_\infty$ in \rf{u_infty}.
We note that this selective long-time asymptotic behaviour can not be easily identified in Eulerian formulation \rf{SSM}, because
the set of the steady states to \rf{SSM} consists of pairs $(h_\infty,\,0)$ with any
sufficiently smooth $h_\infty(x)$ satisfying boundary conditions in \rf{BC}. We are not aware if this long-time solution
behaviour was shown analytically or numerically for the shallow-water system \rf{SSM} in the literature before study~\cite{FKT18}.

In this article, to put this result on a firm basis we prove that solutions to singular heat equation \rf{eq-1}--\rf{eq-3}
considered with $L^2$-source term satisfying \rf{f_mean} exponentially
converge to $u_\infty$ in $H^1$-norm, provided certain integro-algebraic relations between
$\nu,\,f(x,t)$ and $u_0$ hold (see conditions \rf{u0_cond}-\rf{conv_cond} and \rf{R_def}-\rf{nu_+} in Theorems
2.1 and 3.1, respectively). 

The rest of the article is organised as follows. In section 2, we prove an
exponential asymptotic decay result for the time homogeneous case 
$f(x,t)=f_0(x)$ (Theorem 2.1 and Corollary 2.2). In section 3, we generalise it to
the time inhomogeneous case (Theorem 3.1 and Corollary 3.3). In this case the
rate of the asymptotic convergence to the limit profile is determined by the
rate of $L^2$-convergence of the right-hand side $f(\cdot,t)$ as time goes to
infinity. In Examples 2.4 and 3.3 we calculate the explicit form of the
derived asymptotic convergence estimates for two concrete examples of source function
$f$. In particular, Example 2.4 shows how to transform these estimates
into the corresponding ones for the solutions to viscous shallow-water system
\rf{SSM}-\rf{BC} using Eulerian-to-Lagrangian transformation \rf{x_s}.

Our proofs use essentially the energy dissipation methods, 
see e.g.~\cite{CJMTU01,Va06} and references therein. In particular, Theorem 3.1 extends the results
of~\cite{DK06} on asymptotic exponential convergence to the self-similar
Barenblatt's profiles of solutions to porous-media equation \rf{PME}
for the case $m\geqslant 1$ considered also with a time-dependent source $f(x,t)$. We  also note that, in contrast to system considered in~\cite{DK06},   
\rf{eq-1}--\rf{f_mean} posses conservation of mass \rf{cons_mass} motivated by their relation to system \rf{SSM}
via transformation \rf{x_s} discussed above. 

Finally, we note that establishing a rigorous correspondence between systems \rf{SSM} 
and \rf{eq-1}--\rf{f_mean} with special right-hand side \rf{f_def} would demand analysis of regularity for their weak solutions~\cite{KLN11} and is out of the scope of this study.

\section{Time homogeneous case}\label{B}
In the case of $f(x,t)=f_0(x)\in L^2(0,\,1)$ in \rf{eq-1}--\rf{f_mean}, we prove the following result.
\begin{theorem}\label{Th1}
Let
\be
0< u_0(x) \in H^1(0,\,1):\ R_0:=||(u_0^{-1})_x||_2<\infty,
\lb{u0_cond}
\ee
and
\begin{equation}
f_0(x) \in L^2(0,\,1):\,
\int_0^1{f_0(x)\,dx} = 0\quad\text{and }P_0:=||\int_0^x { f_0(s)\, ds}||_2<\infty.
\lb{f_cond}
\end{equation}
If, additionally,
\be
R_0\in[0,\,1)\quad\text{and}\quad\nu\in \left(\frac{2P_0}{1-R_0},+\infty\right)
\lb{conv_cond}
\ee
then a weak solution $u(x,t)$ to problem \rf{eq-1}--\rf{eq-3} considered with \rf{cons_mass} satisfies
\begin{equation}\label{b-00}
0 <\frac{\nu}{\nu(1+R_0)+2P_0}\leqslant u(x,t) \leqslant \frac{\nu}{\nu(1-R_0)-2P_0}<\infty
\quad\text{for all }(x,\,t)\in Q_T.
\end{equation}
Moreover, $u^{-1}(\cdot,t)$ converges to the inverse of steady state solution
$u_{\infty}^{-1}$ defined by \rf{LP}-\rf{C} considered with $f=f_0$, namely,
\be
||u^{-1}(\cdot,t)-u^{-1}_{\infty}||_{H^1}\leqslant||u_0^{-1}-u^{-1}_{\infty}||_{H^1}
\exp\left\{-\frac{\pi^2t}{\nu}[\nu(1-R_0)-2P_0]^2\right\}\ \text{for all }t>0.
\lb{H1_conv}
\ee
\end{theorem}
\begin{proof}[Proof of Theorem~\ref{Th1}]
Introducing a new function $q$ by
\be
\lb{q}
q(x,t) = \nu^{\frac{1}{2}} u^{-1}(x,t),
\ee
we can rewrite problem (\ref{eq-1})--(\ref{eq-3}) in the form
\begin{equation}\label{eq-1-0}
q_t = q^{2} ( q_{xx}  - \nu^{-\frac{1}{2}} f_0) \text{ in } Q_T,
\end{equation}
\begin{equation}\label{eq-2-0}
q_x(0,t)=q_x(1,t)=0\quad\text{for }t>0,
\end{equation}
\begin{equation}\label{eq-3-0}
q(x,0) = q_0(x):=\nu^{\frac{1}{2}}  u^{-1}_0(x)\quad\text{for }x\in(0,\,1).
\end{equation}
Moreover, by \rf{cons_mass} we have
\begin{equation}\label{mass-0}
 \int_0^1{q^{-1} \,dx} = \int_0^1{q^{-1}_0 \,dx}=\nu^{-\frac{1}{2}}.
\end{equation}
Multiplying (\ref{eq-1-0}) by $- ( q_{xx}  - \nu^{-\frac{1}{2}} f_0) $, after integration over $(0,\,1)$, we get
\begin{equation}\label{eq-5-0}
 \tfrac{d}{dt}\mathcal{E}(q(t)) +
 \int_0^1{q^2 (q_{xx}  - \nu^{-\frac{1}{2}} f_0(x) )^2 \,dx} =0,
\end{equation}
where the {\it energy functional} is
$$
\mathcal{E}(q(t)):=\int_0^1{ [ \tfrac{1}{2} q_x^2 + \nu^{-\frac{1}{2}} f_0(x) q ]\,dx}.
$$
Integrating (\ref{eq-5-0}) in time, one obtains the energy equality:
\begin{equation}\lb{EnEq}
\mathcal{E}(q(t)) +
 \iint \limits_{Q_T}{q^2 (q_{xx}  - \nu^{-\frac{1}{2}} f_0(x) )^2 \,dx dt} = \mathcal{E}(q_0).
\end{equation}
Let us introduce a set
$$
\mathcal{M}_\nu : = \Bigl \{q \in H^1(0,\,1):\ q_x(0)=q_x(1)=0,\ \int_0^1{q^{-1} \,dx}=\nu^{-\frac{1}{2}}\Bigr\}.
$$
Let also $q_{\min}(x)$ be a unique solution of the corresponding Euler-Lagrange equation
$$
q_{xx} = \nu^{-\frac{1}{2}}  f_0(x)\quad\text{with }q_x(0)=q_x(1)=0,
$$
describing critical points of the functional $\mathcal{E}$ over $\mathcal{M}_\nu$.
Hence, 
\be
q_{\min}(x)=\frac{\nu^\frac{1}{2}}{u_\infty(x)}\quad\text{with }u_\infty(x)\text{ defined in \rf{LP}-\rf{C}}.
\ee
Next, we look for a lower bound for $\mathcal{E}(q(t))$.
Using integration by parts, one can estimate
\bes
\int_0^1f_0(x)q\,dx=-\int_0^1q_{\min,x}q_{x}\,dx\leqslant ||q_{\min,x}||_2||q_x||_2.
\ees
Hence, one has
\bes
\mathcal{E}(q(t))\geqslant\frac{1}{2}||q_x||_2^2-||q_{\min,x}||_2||q_x||_2\geqslant-\frac{1}{2}||q_{\min,x}||_2=\min_{q\in\mathcal{M}_\nu}\mathcal{E}(q).
\ees
We conclude that the lower semi-continuous bounded from below functional $\mathcal{E}(q)$ attains its unique global minimum $q_{\min}(x)$ on $\mathcal{M}_\nu$.
According to definitions \rf{u0_cond}-\rf{f_cond},
\bes
P_0=\nu^{1/2}||q_{\min,x}||_2\quad\text{and }R_0=\nu^{-1/2}||q_{0,x}||_2.
\ees
Next, \rf{EnEq} together with H\"older inequality imply
\bes
0\geqslant\frac{1}{2}||q_x(\cdot,t)||_2^2-\nu^{-1/2}P_0||q_x(\cdot,t)||_2-\mathcal{E}(q_0),
\ees
and, consequently, the following upper bound for solutions to
\rf{eq-1-0}-\rf{eq-3-0}:
\bes
||q_x||_2 \leqslant\nu^{-1/2}P_0 + \sqrt{  \nu^{-1} P_0^2  + 2 E(q_0)  }\leqslant\nu^{-1/2}P_0 + \sqrt{\nu^{-1} P_0^2  + \nu R_0^2 + 2 P_0 R_0},
\ees
whence
\be
||q_x(\cdot,t)||_2\leqslant 2\nu^{-1/2}P_0+\nu^{1/2}R_0\quad\text{for all }t\in(0,\,T).
\lb{UpBnd}
\ee
By (\ref{mass-0}) there exists $x_0 \in (0,\,1)$ such that $q(x_0,t) =\nu^{\frac{1}{2}}$. 
Using this and \rf{UpBnd} one estimates:
\bes
\left|q(x,t)-\nu^{1/2}\right|\leqslant||q_x(\cdot,t)||_2\leqslant 2\nu^{-1/2}P_0+\nu^{1/2}R_0
\ees
The last inequality implies the pointwise estimates
\be
\nu^{1/2}-2\nu^{-1/2}P_0-\nu^{1/2}R_0\leqslant q(x,t)\leqslant\nu^{1/2}+2\nu^{-1/2}P_0+\nu^{1/2}R_0
\lb{Ms2}
\ee
for all $(x,t)\in Q_T$. From \rf{Ms2} it follows that $q(x,t)>0$ if conditions \rf{conv_cond} holds. 
Furthermore, by \rf{q} and \rf{Ms2} conditions \rf{conv_cond} imply estimate \rf{b-00}.

Next, let us introduce 
\be 
w:= q - q_{\min}.
\lb{w}
\ee
Then
\begin{multline*}
\mathcal{ E }(q|q_{\min}) := \mathcal{ E }(q(t)) - \mathcal{E}(q_{\min}) = \tfrac{1}{2} \int_0^1{ w_x ( q  + q_{\min } )_x \,dx}
+ \\
\nu^{-\frac{1}{2}} \int_0^1{  f_0(x) w\,dx} =
- \tfrac{1}{2} \int_0^1{ w ( q  + q_{\min } )_{xx} \,dx}
+   \int_0^1{  q_{\min,xx}   w\,dx} = \\
 - \tfrac{1}{2} \int_0^1{ w w _{xx} \,dx} =
  \tfrac{1}{2} \int_0^1{ w_x^2 \,dx}.
\end{multline*}
From here we find that
\begin{equation}\label{energy}
\tfrac{d}{dt}\mathcal{ E }(q|q_{\min}) = - \int_0^1{ w_t w_{xx} \,dx} =
- \int_0^1{q^2 w^2_{xx} \,dx} \leqslant 0.
\end{equation}
By \rf{eq-2-0} the Poincar\'{e} inequality
\bes
\int_0^1{ w_x^2 \,dx} \leqslant \tfrac{1}{\pi^2}    \int_0^1{ w^2_{xx} \,dx} 
\ees
holds. Combining it with \rf{b-00}, one arrives at
$$
\frac{\pi^2}{\nu}[\nu(1-R_0)-2P_0]^2\mathcal{ E }(q|q_{\min}) \leqslant \int_0^1{q^2 w^2_{xx} \,dx}.
$$
Hence, by (\ref{energy}) we deduce that
\begin{equation}\label{eq-6-1}
 \tfrac{d}{dt}\mathcal{ E }(q|q_{\min}) +\frac{\pi^2}{\nu}[\nu(1-R_0)-2P_0]^2\mathcal{ E }(q|q_{\min})\leqslant 0 .
\end{equation}
From (\ref{eq-6-1}) it follows that
$$
0 \leqslant \mathcal{ E }(q|q_{\min}) \leqslant 
\mathcal{ E }(q_0|q_{\min})\exp\left\{-\frac{\pi^2}{\nu}[\nu(1-R_0)-2P_0]^2t\right\}\to 0
\text{ as } t \to +\infty. 
$$
Therefore, we obtain that
$$
 q(\cdot,t) \to  q_{\min} \text{ strongly  in } H^1(0,1) \text{ as } t \to +\infty
$$
and, consequently, \rf{H1_conv}.
\end{proof}
\begin{corollary}\label{Col}
Conditions of Theorem 2.1 imply also the direct exponential asymptotic convergence of 
$u(\cdot,t)$ to $u_{\infty}$ in $H^1(0,\,1)$. Namely, there exists a constant $C$ depending on $\nu,\,f_0(x)$ and $u_0$ such that 
\bes
||u(\cdot,t)-u_{\infty}||_{H^1}\leqslant C||u_0-u_{\infty}||_{H^1}
\exp\left\{-\frac{\pi^2t}{\nu}[\nu(1-R_0)-2P_0]^2\right\}\ \text{for all }t>0.
\ees
\end{corollary}
\begin{proof}
The statement follows from \rf{H1_conv}, uniform bounds \rf{b-00} and the following pointwise bounds:
\bes
\left(1+\frac{P_0}{\nu}\right)^{-1}\leqslant u_\infty(x)\leqslant \left(1-\frac{P_0}{\nu}\right)^{-1},\
\left(1+R_0\right)^{-1}\leqslant u_0(x)\leqslant \left(1-R_0\right)^{-1}\ \text{for }x\in (0,\,1).
\ees
In the last estimates we used that $R_0<1$ and $P_0<\nu$ holding by assumption \rf{conv_cond}.
\end{proof}
\begin{example}
Here, we show how the asymptotic convergence result from Theorem 2.1 can
be applied to show the corresponding convergence (at least formally) of solutions to the viscous
shallow water system \rf{SSM}-\rf{BC}. The following example is similar to one
presented in~\cite[Remark 4.3 and Fig. 2]{FKT18} and demonstrates that
a solution to \rf{SSM} starting from a constant $h_0(y)=1$ but having
non-zero initial velocity $v_0(y)$ asymptotically, as time goes to infinity, converges to an inhomogeneous limit profile
$h_\infty(y)$.

Firstly, note that Eulerian-to-Lagrangian transformation \rf{x_s}
together with relation \rf{u_h_rel} allows to rewrite the asymptotic convergence
estimate \rf{H1_conv} of Theorem 2.1 as one for solution $h(y,t)$ to
\rf{SSM}-\rf{BC} in the form:
\be
||h(y(\cdot,t),t)-h_{\infty}(y_\infty(\cdot))||_{H^1}\leqslant||h_0(y(\cdot,0))-h_{\infty}(y_\infty(\cdot))||_{H^1}
\exp\left\{-\frac{\pi^2t}{\nu}[\nu(1-R_0)-2P_0]^2\right\},
\lb{H1_conv_h}
\ee
holding for all $t>0$, where
\bes
h_{\infty}(y_\infty(x))=\frac{M}{u_\infty(x)},\ y_\infty(x)=\int_0^xu_\infty(x)\,dx,\quad\text{and}\quad h_0(y(x,0))=\frac{M}{u_0(x)}
\ees
are the limit and initial height profiles, respectively.
Moreover, by definitions \rf{u0_cond}-\rf{f_cond} and again relation
\rf{u_h_rel} one finds that
\bes
R_0=\tfrac{1}{M}||h_{0,x}||_2\quad\text{and }P_0=\tfrac{\nu}{M}||h_{\infty,x}||_2.
\ees
Using these, one can write out the right-hand side of \rf{H1_conv_h} solely
using $h_\infty$ and $h_0$ functions as
\begin{multline}
||h(y(\cdot,t),t)-h_{\infty}(y_\infty(\cdot))||_{H^1}\\\leqslant||h_0(y(\cdot,0))-h_{\infty}(y_\infty(\cdot))||_{H^1}
\exp\left\{-\pi^2t\left[1-\tfrac{||h_{0,x}||_2+2||h_{\infty,x}||_2}{M}\right]^2\right\}\ \text{for all }t>0.
\lb{H1_conv_h1}
\end{multline}
Interestingly, the exponential asymptotic decay rate of \rf{H1_conv_h1} turns out to be
independent of viscosity $\nu$ in this case.

Next, let us check conditions of Theorem 2.1 and calculate all entries of \rf{H1_conv_h1} for the following
concrete example of the initial data to \rf{SSM}:
\bea
&&h_0(y)=h_0(y(x,0))=u_0(x)=1,\nonumber\\
&&v_0(y)=\tfrac{1}{2}\sin(\pi y),\ v_0(y(x,0))=\tfrac{1}{2}\sin(\pi x),\nonumber\\
&&M=\nu=1,
\lb{ex_init_data}
\eea
having a constant initial height profile.

In this case, according to definitions \rf{f_def} and \rf{u_h_rel}
\bes
f_0(x)=\frac{M}{h_0(y(x,0))}\left[v_0(y(x,0))\right]_y=[v_0(y(x,0))]_x=\tfrac{\pi}{2}\cos(\pi x),
\ees
where we again used transformation of variables \rf{x_s}. Hence,
by definitions \rf{u0_cond}-\rf{f_cond} one calculates:
\be
R_0=0\quad\text{and }P_0=||\int_0^xf_0(s)\,ds||_2=||v_0(y(x,\,0))||_2=\frac{1}{2\sqrt{2}}.
\lb{Ms_ex1}
\ee
Therefore, in this case
\bes
\nu(1-R_0)-2P_0=1-\tfrac{1}{\sqrt{2}}>0
\ees
and all conditions of Theorem 2.1 are satisfied.
From equation \rf{eq-1} and definition \rf{u_h_rel} one obtains
\bes
h_\infty(y_\infty(x))-h_\infty(0)=\int_0^xv_0(y(s,\,0))\,ds=\frac{1}{2\pi}(1-\cos(\pi x)).
\ees
In turn, constant $h_\infty(0)$ is determined by the analogue of condition
\rf{C}:
\bes
\int_0^1\frac{dx}{h_\infty(y_\infty(x))}=1,
\ees
so that one obtains
\be
h_\infty(y_\infty(x))=\frac{1}{2\pi}(C_\infty-\cos(\pi x)),\text{ where }C_\infty\approx 6.37.
\lb{Ms_ex2}
\ee
One can also calculate explicitly
\be
||h_\infty(y_\infty(\cdot))-h_0(y(\cdot,0))||_{H^1}=||\frac{1}{2\pi}(C_\infty-\cos(\pi x))-1||_{H^1}\approx 0.37.
\lb{Ms_ex3}
\ee
Finally, substituting \rf{Ms_ex1}-\rf{Ms_ex3} into \rf{H1_conv_h} one
obtains an explicit asymptotic decay estimate for the solution to \rf{SSM} with initial
data given by \rf{ex_init_data}:
\beas
||h(y(\cdot,t),t)-\frac{1}{2\pi}(C_\infty-\cos(\pi\cdot))||_{H^1}&\leqslant&||\frac{1}{2\pi}(C_\infty-\cos(\pi x))-1||_{H^1} 
\exp\left\{-\pi^2t\left[1-\tfrac{1}{\sqrt{2}}\right]^2\right\}\\
&\leqslant& 0.37\exp\left\{-0.84t\right\}\quad \text{for all }t>0,
\eeas
We note that the last estimate implies also an asymptotic convergence result of
$h(\cdot,t)$ to $h_\infty$ in original Eulerian coordinates $(y,t)$ as $t\go\infty$.
\end{example}

\section{Time inhomogeneous case}\label{B2}
In the case of $f\in C(0,T;L^2(0,\,1))$ for all $T>0$ in \rf{eq-1}--\rf{f_mean}, we prove the following result.
\begin{theorem}\label{Th2}
Let
\be
0<u_0(x)\in H^1(0,\,1):\ R_0:=||(u_0^{-1})_x||_2<\infty.
\lb{R_def}
\ee
Let $f\in C(0,T;L^2(0,\,1))$ for all $T>0$ and $f_0,\,f_\infty\in L^2(0,1)$ be such that
\be
||f(\cdot,t)-f_\infty||_2\go 0\text{ as }t\go\infty\text{ and }||f(\cdot,t)-f_0||_2\go 0\text{ as }t\go0.
\lb{f_conv}
\ee
Assume that
\be
P_0:=|| \int_0^x {f_0(s)\,ds}||_2<\infty,\quad N_\infty:=\int_0^\infty||\int_0^xf_t(s,\,t)\,ds||_2\,dt<\infty.
\lb{P_def}
\ee
If conditions
\be
R_0\in[0,\,1)\text{ and }\nu\in\left(\nu_+,+\infty\right)
\lb{conv_cond_nonst}
\ee
hold with
\be
\nu_+:=\frac{2N_\infty+(1+R_0)P_0+\sqrt{(2N_\infty+P_0(1+R_0))^2-N_\infty^2(1-R_0^2)}}{1-R_0^2},
\lb{nu_+}
\ee
then a weak solution $u(x,t)$ to \rf{eq-1}--\rf{eq-3} considered with
\rf{cons_mass} obeys the uniform bounds
\be
0<A_-\leqslant u(x,t)\leqslant A_+<\infty
\lb{b_nonst}
\ee
with
\be
A_\pm:=\frac{\nu}{\nu\mp\left[2N_\infty+P_0+\sqrt{(N_\infty+P_0)^2+2(N_\infty+P_0)N_\infty+\nu^2R_0^2+2\nu P_0R_0}\right]}
\lb{A_pm}
\ee
for all $x\in (0,1)$ and $t>0$.

Furthermore, there exists a unique solution $u_\infty\in H^2(0,1)$ of the limit equation
\be
-(u_\infty^{-2}u_{\infty,x})_x=\nu^{-1}f_\infty(x)\text{ with } u_{\infty,x}(0)=u_{\infty,x}(1)=0,
\lb{lim_profile}
\ee 
given by \rf{LP}-\rf{C} considered with $f=f_\infty$, such that $u^{-1}(\cdot,t)$ converges
to $u_{\infty}^{-1}$, namely, for all $t>0$ it holds
\be
||u^{-1}(\cdot,t)-u^{-1}_{\infty}||_{H^1}\leqslant
e^{-B\,t}\left[||u_0^{-1}-u^{-1}_{\infty}||_{H^1}+
C\int \limits_0^t {\|f(.,s)  - f_{\infty}\|_2^2 e^{B\,s} ds }\right]\go 0,
\lb{H1_conv_nonst}
\ee
as $t\go\infty$, where constants
\be
C:=\tfrac{1}{\nu^{1/2}A_-^2}+\tfrac{1}{\nu^{1/2}A_+^2},\quad B:=\tfrac{\nu\pi^2}{2A_+^2}.
\lb{B_const}
\ee
\end{theorem}
\begin{proof}[Proof of Theorem~\ref{Th2}]
Similar to the proof of Theorem~\ref{Th1} by introducing a new function $q$ as
\be
\lb{q_1}
q(x,t) = \nu^{\frac{1}{2}} u^{-1}(x,t),
\ee
we can rewrite problem (\ref{eq-1})--(\ref{eq-3}) in the form
\begin{equation}\label{eq-1-1}
q_t = q^{2} ( q_{xx}  - \nu^{-\frac{1}{2}} f) \text{ in } Q_T,
\end{equation}
\begin{equation}\label{eq-2-1}
q_x(0,t)=q_x(1,t)=0\quad\text{for }t>0,
\end{equation}
\begin{equation}\label{eq-3-1}
q(x,0) = q_0(x):=\nu^{\frac{1}{2}}  u^{-1}_0(x)\quad\text{for }x\in(0,\,1),
\end{equation}
equipped with
\begin{equation}\label{mass-1}
 \int_0^1{q^{-1}(\cdot,\,t)\,dx} = \int_0^1{q^{-1}_0 \,dx}=\nu^{-\frac{1}{2}}\quad\text{for }t>0.
\end{equation}
Multiplying (\ref{eq-1-1}) by $- ( q_{xx}  - \nu^{-\frac{1}{2}} f) $,
after integration by parts over $Q_T$ and using property \rf{f_mean}, one
obtains the following energy equality:
\begin{equation}\label{eq-5-1}
\mathcal{E}(q(t))+\int_0^t\int_0^1{q^2 (q_{xx}  - \nu^{-\frac{1}{2}} f)^2 \,dx}+\nu^{-1/2}\int_0^t\int_0^1q_x\left[\int_0^xf_\tau\,ds\right]\,dxd\tau=\mathcal{E}(q_0),
\end{equation}
with the {\it energy functional}
\be
\mathcal{E}(q(\cdot,t)):=\int_0^1{ [ \tfrac{1}{2} q_x^2 + \nu^{-\frac{1}{2}} fq]\,dx}.
\lb{en_def}
\ee
Applying the H\"older inequality in \rf{eq-5-1} gives
\be
||q_x(\cdot,\,t)||_2^2\leqslant 2\nu^{-1/2}P(t)||q_x(\cdot,\,t)||_2+2\nu^{-1/2}\int_0^t||q_x(\cdot,\,\tau)||_2N'(\tau)\,d\tau+2\mathcal{E}(q_0),
\lb{y_t1}
\ee
where we introduced functions
\bes
P(t):=||\int_0^x { f(s,t)\,ds}||_2,\quad
N(t):=\int_0^t||\int_0^xf_\tau(s,\,\tau)\,ds||_2\,d\tau.
\ees
Using these, $f\in L^\infty(0,T;L^2(0,\,1))$ and \rf{P_def} one estimates
\beas
2P(t)P'(t)&=&\frac{d}{dt}P^2(t)=\frac{d}{dt}\int_0^1\left[\int_0^xf(s,\,t)\,ds\right]^2\,dx\\
&=&2\int_0^1\left[\int_0^xf(s,\,t)\,ds\right]\left[\int_0^xf_t(s,\,t)\,ds\right]\,dx\\
&\leqslant& 2P(t)\left(\int_0^1\left[\int_0^xf_t(s,\,t)\,ds\right]^2dx\right)^{1/2}=2P(t)N'(t).
\eeas
Hence, one deduces bounds
\be
P'(t)\leqslant N'(t)\quad\text{and}\quad P(t)\leqslant N_\infty+P_0,
\lb{PN_bnd}
\ee
where we used definitions \rf{P_def}.
Using \rf{PN_bnd}, one can rewrite \rf{y_t1} as
\bes
\left(||q_x(\cdot,\,t)||_2-\nu^{-1/2}(N_\infty+P_0)\right)^2\leqslant\nu^{-1}(N_\infty+P_0)^2+2\nu^{-1/2}\int_0^t||q_x(\cdot,\,\tau)||_2N'(\tau)d\tau+2\mathcal{E}(q_0).
\ees
By introduction of the new nonnegative function
\be
v(t):=(||q_x(\cdot,\,t)||_2-\nu^{-1/2}(N_\infty+P_0))^2,
\lb{v_def}
\ee
the last inequality together with \rf{PN_bnd} and \rf{en_def} implies:
\beas
v(t)&\leqslant&\nu^{-1}(N_\infty+P_0)^2+2\nu^{-1/2}\int_0^tv^{1/2}(\tau)N'(\tau)\,d\tau+2\nu^{-1}\int_0^t(N_\infty+P_0)N'(\tau)\,d\tau+2\mathcal{E}(q_0)\\
&=& \nu^{-1}(N_\infty+P_0)^2+2\nu^{-1/2}\int_0^tv^{1/2}N'(\tau)\,d\tau+2\nu^{-1}(N_\infty+P_0)N(t)+2\mathcal{E}(q_0)\\
&\leqslant& a_0+2\nu^{-1/2}\int_0^tv^{1/2}(\tau)N'(\tau)\,d\tau,
\eeas
where
\bes
a_0:=\nu^{-1}(N_\infty+P_0)^2+2\nu^{-1}(N_\infty+P_0)N_\infty+\nu R_0^2+2P_0R_0,
\ees
with $R_0$ being defined in \rf{R_def}.

The last inequality after an application of Bihari-LaSalle
lemma~\cite{Bi56} to $v(t)$, while using $a_0\geqslant 0$ and $N'(t)\geqslant 0$, implies 
\bes
v^{1/2}(t)\leqslant a_0^{1/2}+\nu^{-1/2}N(t)\leqslant a_0^{1/2}+\nu^{-1/2}N_\infty,
\ees
which, in turn, using definition \rf{v_def} gives
\bes
||q_x(\cdot,t)||_2\leqslant a_0^{1/2}+2\nu^{-1/2}N_\infty+\nu^{-1/2}P_0.
\ees
In turn, this together with \rf{mass-1} implies
\beas
|q(x,t)-\nu^{1/2}|&\leqslant&\int_0^1|q_x(x,t)|dx\leqslant ||q_x(\cdot,t)||_2\\
&\leqslant& a_0^{1/2}+2\nu^{-1/2}N_\infty+\nu^{-1/2}P_0 ,
\eeas
and, consequently, 
\be
\nu^{1/2}-a_0^{1/2}-2\nu^{-1/2}N_\infty-\nu^{-1/2}P_0\leqslant
q(x,t)\leqslant \nu^{1/2}+a_0^{1/2}+2\nu^{-1/2}N_\infty+\nu^{-1/2}P_0
\lb{Ms4}
\ee
hold for all $x\in(0,1)$ and $t>0$. Estimates \rf{Ms4} together with definition \rf{q_1} imply
the uniform pointwise bounds \rf{b_nonst}-\rf{A_pm} provided
\bes
\nu-a_0^{1/2}\nu^{1/2}-2N_\infty-P_0>0.
\lb{pos_cond}
\ees
This is equivalent to the following system of inequalities:
\beas
\nu&>&2N_\infty+P_0,\\
1&>&R_0^2,\\
(1-R_0^2)\nu^2-2\nu(2N_\infty+P_0+R_0P_0)+N_\infty^2&>&0.
\eeas
By solving the last inequality explicitly this system turns into:
\beas
\nu&>&2N_\infty+P_0,\\
1&>&R_0^2,\\
\nu&\in& (0,\,\nu_-)\cup(\nu_+,+\infty),
\eeas
where
\bes
\nu_\pm:=\frac{2N_\infty+(1+R_0)P_0\pm\sqrt{(2N_\infty+P_0(1+R_0))^2-N_\infty^2(1-R_0^2)}}{1-R_0^2},
\ees
It is easy to check that $\nu_-<2N_\infty+P_0$ for all nonnegative $N_\infty,\,P_0$
and $R_0$. Therefore, the last system of inequalities is equivalent to
conditions \rf{conv_cond_nonst}-\rf{nu_+}.

Next, let $u_\infty\in H^2(0,1)$ be a unique solution of the limit equation \rf{lim_profile}, and
\be
q_\infty(x):=\frac{\nu^{1/2}}{u_\infty(x)}.
\lb{w_1}
\ee
Let $w(x,t):=q(x,t)-q_\infty(x)$. Then
\beas
\mathcal{ E }(q|q_\infty)&:=&
\mathcal{ E }(q(t)) - \mathcal{E}(q_\infty) = \tfrac{1}{2} \int_0^1{ w_x ( q  + q_{\min } )_x \,dx}
+\nu^{-\frac{1}{2}} \int_0^1{ (q f(x,t) - q_\infty(x) f_{\infty}(x) )\,dx}\\
&=&- \tfrac{1}{2} \int_0^1{ w ( q  + q_{\min } )_{xx} \,dx}
+\nu^{-\frac{1}{2}}\int_0^1{w \, f_{\infty}(x)\,dx} +
 \nu^{-\frac{1}{2}}\int_0^1{q [f  - f_{\infty}(x)]\,dx}\\
&=&- \tfrac{1}{2} \int_0^1{ w w _{xx} \,dx} +
\nu^{-\frac{1}{2}}\int_0^1{q [f  - f_{\infty}(x)]\,dx}  = 
\tfrac{1}{2} \int_0^1{ w_x^2 \,dx} +\nu^{-\frac{1}{2}}\int_0^1{q [f  - f_{\infty}(x)]\,dx}.
\eeas
From here we find that
\bea 
\tfrac{d}{dt}\mathcal{ E }(q|q_\infty)&=& 
- \int_0^1{ w_t [ w_{xx} - \nu^{-\frac{1}{2}} (f  - f_{\infty}(x)) ] \,dx}
+ \nu^{-\frac{1}{2}}\int_0^1{q f_t  \,dx}\nonumber\\
&=&-\int_0^1{q^2 [ w_{xx} - \nu^{-\frac{1}{2}} (f  - f_{\infty}(x)) ]^2 \,dx}
+ \nu^{-\frac{1}{2}}\int_0^1{q f_t  \,dx} ,
\label{energy1}
\eea
i.\,e.
\begin{multline*}
\tfrac{1}{2} \tfrac{d}{dt} \int_0^1{ w_x^2 \,dx} +\int_0^1{q^2 [ w_{xx} - \nu^{-\frac{1}{2}} (f  - f_{\infty}(x)) ]^2 \,dx} = \\    
- \nu^{-\frac{1}{2}}\int_0^1{q_t [f  - f_{\infty}(x)]  \,dx}  \leqslant   
 \tfrac{1}{2} \int_0^1{q^2 [ w_{xx} - \nu^{-\frac{1}{2}} (f  - f_{\infty}(x)) ]^2 \,dx} + \\
\tfrac{\nu^{-1}}{2}\int_0^1{q^2 [f  - f_{\infty}(x)]^2  \,dx}, 
\end{multline*}
whence
\begin{equation}\label{start}
\tfrac{d}{dt} \int_0^1{ w_x^2 \,dx} +\int_0^1{q^2 [ w_{xx} - \nu^{-\frac{1}{2}} (f  - f_{\infty}(x)) ]^2 \,dx}
\leqslant  \nu^{-1} \int_0^1{q^2 |f  - f_{\infty}(x)|^2  \,dx}.  
\end{equation}
Let us denote by
$$
Z(x,t):= w_x - \nu^{-\frac{1}{2}} \int \limits_{0}^x (f  - f_{\infty}(s))\,ds,\quad K(x,t):=\int \limits_{0}^x (f  - f_{\infty}(s))\,ds.
$$
As $Z(0,t)=Z(1,t)=K(0,t)=K(1,t)=0$ for $t>0$, then by the Poincar\'{e} inequality and \rf{b_nonst}, one has
\be
 \tfrac{\nu \pi^2}{A_+^2} \int_0^1{Z^2 \,dx} \leqslant \int_0^1{ q^2 Z_x^2\,dx},\quad \pi||K(\cdot,\,t)||_2\leqslant||K_x(\cdot,\,t)||_2=||f(\cdot,\,t)  - f_{\infty}||_2.
\lb{Poincare}
\ee
Taking into account
\beas
\int_0^1{Z^2 \,dx}&=&\int_0^1{ w_x^2 \,dx} - 2 \nu^{-\frac{1}{2}}\int_0^1{ w_xK\,dx} +  
 \nu^{-1} \int_0^1K^2 \,dx\\
&\geqslant&\tfrac{1}{2}\int_0^1w_x^2 \,dx-\nu^{-1} \int_0^1K^2 \,dx,
\eeas
and \rf{Poincare} we deduce that
\bes
\int_0^1{ v^2[ w_{xx} - \nu^{-\frac{1}{2}} (f  - f_{\infty}(x)) ]^2 \,dx}
\geqslant \tfrac{\nu \pi^2}{A_+^2} \Bigl[\tfrac{1}{2}\int_0^1{ w_x^2 \,dx}-\tfrac{1}{\pi^2\nu}\| f(.,t)  - f_{\infty}(x) \|_2^2\Bigr].
\ees
Next, using again \rf{b_nonst} one estimates
\begin{multline} \label{jen}
\nu^{-1} \int_0^1{q^2 |f  - f_{\infty}(x)|^2  \,dx}+\tfrac{1}{A_+^2}\int_0^1\| f(.,t)  - f_{\infty}(x) \|_2^2 
\leqslant \\
\left(\tfrac{1}{A_-^2}+\tfrac{1}{A_+^2}\right)\| f(.,t)  - f_{\infty}(x) \|_2^2:=D(t).
\end{multline}
So, by (\ref{start}) one obtains
\bes
\tfrac{d}{dt}\int_0^1{ w_x^2 \,dx}  +\tfrac{\nu \pi^2}{2A_+^2}\int_0^1{ w_x^2 \,dx} \leqslant D(t). 
\ees
Consequently, it follows that
\be
\int_0^1{ w_x^2 \,dx}\leqslant e^{-B\,t}\Bigl[\| w_{0,x} \|_2^2 +\int \limits_0^t { D(s) e^{B\,s} ds }\Bigr]
\lb{Ms5}
\ee
where $B$ is from \rf{B_const}. It remains only to show that the right-hand side
of \rf{Ms5} tends to zero as $t\go\infty$. This is true immediately, if
\bes
\int\limits_0^\infty { D(s) e^{B\,s} ds }<\infty.
\ees
But if
\bes
\displaystyle \lim_{t\go\infty}\int\limits_0^t { D(s) e^{B\,s} ds }=+\infty,
\ees
then after an application of the L'Hopital's rule one has:
\bes
\displaystyle \lim_{t\go\infty}e^{-B\,t}\int\limits_0^t { D(s) e^{B\,s} ds }=\lim_{t\go\infty}\tfrac{D(t)}{B}=0.
\ees
Therefore, we conclude that \rf{Ms5} together with \rf{w_1} and \rf{q_1} implies asymptotic
convergence estimate \rf{H1_conv_nonst}.
\end{proof}
\begin{remark}
One can check that results of Theorem 3.1 are consistent with those of Theorem
2.1. Indeed, in the time-homogeneous case $f(x,t)=f_0(x)$ formulae
\rf{conv_cond_nonst}-\rf{nu_+} and \rf{b_nonst}-\rf{A_pm} coincide with
\rf{conv_cond} and \rf{b-00}, respectively. This is verified by the direct
substitution of $N_\infty=0$ into the former formulae. Also the asymptotic
convergence estimate \rf{H1_conv_nonst}-\rf{B_const} is consistent with 
\rf{H1_conv} in this case.
\end{remark}
\begin{corollary}\label{Col1}
Conditions of Theorem 3.1 imply also the direct exponential asymptotic convergence of 
$u(\cdot,t)$ to $u_{\infty}$ in $H^1(0,1)$. Namely, there exists a constant $C_1$ depending on $\nu,\,f(x,t)$ and $u_0,\,f_0,\,f_\infty$ such that 
for all $t>0$ it holds
\bes
\hspace{-1.2cm}||u(\cdot,t)-u_\infty||_{H^1}\leqslant
C_1e^{-B\,t}\left[||u_0-u_\infty||_{H^1}+\int \limits_0^t {\|f(.,t)  - f_{\infty}\|_2^2 e^{B\,s} ds }\right]\go 0\text{ as }t\go\infty,
\ees
\end{corollary}
\begin{proof}
The statement follows from \rf{H1_conv_nonst}, uniform bounds \rf{b_nonst} and the following pointwise bounds:
\bea
&&\left(1+\frac{N_\infty+P_0}{\nu}\right)^{-1}\leqslant u_\infty(x)\leqslant \left(1-\frac{N_\infty+P_0}{\nu}\right)^{-1},\nonumber\\
&&\left(1+R_0\right)^{-1}\leqslant u_0(x)\leqslant \left(1-R_0\right)^{-1}\ \text{for }x\in (0,\,1).
\lb{u_inf_bnd}
\eea
Note, that the pointwise bounds for $u_\infty$ follow from the estimate
\bes
|u_\infty^{-1}(x)-1|\leqslant||\left(u_\infty^{-1}\right)_x||_2\leqslant\nu^{-1}||\int_0^xf_\infty(s)\,ds||_2\quad\text{for all }x\in(0,\,1),
\ees
and \rf{f_conv}, \rf{PN_bnd}.
In \rf{u_inf_bnd} we also used that $R_0<1$ and $N_\infty+P_0<\nu$ hold by assumptions \rf{conv_cond_nonst}-\rf{nu_+}.
\end{proof}
\begin{example}
Here, we calculate an explicit form of the asymptotic convergence estimate \rf{H1_conv_nonst}-\rf{B_const} for
the particular right-hand side
\bes
f(x,t)=\min\{1,\,1/t\}\cos(\pi x),\quad ||f(\cdot,t)||_2\go 0\text{ as }t\go\infty;
\ees
and the initial data
\bes
u_0=u_\infty=1.
\ees
In this case,
\beas
&&R_0=0,\quad P_0=||\int_0^x\cos(\pi s)\,ds||_2=\tfrac{1}{\sqrt{2}},\\
&&N_\infty=\int_1^\infty||\int_0^x\tfrac{\cos(\pi s)}{t^2}\,ds||_2\,dt=\tfrac{1}{\sqrt{2}}.
\eeas
Therefore, conditions \rf{conv_cond_nonst}-\rf{nu_+} reduce to
\bes
\nu\in\left(\tfrac{3}{\sqrt{2}}+2,\,\infty\right),
\ees
while constants \rf{A_pm} to
\bes
A_\pm=\tfrac{\nu}{\nu\mp[3/\sqrt{2}+2]}.
\ees
Hence, constants in \rf{B_const} become
\be
C=\tfrac{2}{\nu^{5/2}}(\nu^2+[3/\sqrt{2}+2]^2),\quad B=\frac{\pi^2[\nu-3/\sqrt{2}-2]^2}{2\nu}.
\lb{C_B_exmp}
\ee
Finally, estimate \rf{H1_conv_nonst} reduces in this example to
\bes
||u^{-1}(\cdot,t)-1||_{H^1}\leqslant
\tfrac{C}{2}e^{-Bt}\left(\int_0^1e^{Bs}\,ds+\int_1^t\tfrac{1}{s^2}e^{Bs}\,ds\right)\text{ for all }t>0.
\ees
For sufficiently large $t>1$ the last estimate implies that
\bes
||u^{-1}(\cdot,t)-1||_{H^1}\leqslant\frac{C}{2}\exp\{-Bt\}\left(\frac{2}{Bt^2}+\frac{e^B-1}{B}\right)\text{ as }t\go\infty
\ees
with the exponential decay rate $B$ from \rf{C_B_exmp}.
\end{example}

\section{Acknowledgements}
GK would like to acknowledge support from the Oxford Centre for Industrial and Applied Mathematics (OCIAM) of the University of Oxford. 
Part of this research was performed during participation of one of the authors at the thematic research program "Mathematical Biology"
organised by the Mittag-Leffler Institute.

\bibliographystyle{unsrtnat}
\bibliography{bibliography}
\clearpage
\addtocounter{tocdepth}{2}

\end{document}